\documentclass[12pt,twoside]{amsart}

\usepackage{amsmath}
\usepackage{amssymb}
\usepackage{amscd}
\usepackage{graphics}
\usepackage{graphicx}

\renewcommand{\bar}{\overline}

\newcommand{\PP}{\mathbb{P}}
\newcommand{\QQ}{\mathbb{Q}}

\newcommand{\ZZ}{\mathbb{Z}}

\newcommand{\Qp}{\QQ_p}
\newcommand{\Zp}{\ZZ_p}

\newcommand{\Qpbar}{\bar{\QQ}_p}

\newcommand{\ints}{{\mathcal O}}

\newcommand{\maxid}{{\mathcal M}}

\newcommand{\calM}{{\mathcal M}}

\newcommand{\Kbar}{\overline{K}}

\newcommand{\PKbar}{\PP^1(\overline{K})}
\newcommand{\Pkbar}{\PP^1(\overline{k})}

\DeclareMathOperator{\PGL}{PGL}

\DeclareMathOperator{\red}{red}

\theoremstyle{plain}
\newtheorem{thm}{Theorem}[section]

\newtheorem{lemma}[thm]{Lemma}

\newtheorem*{thmblank}{Theorem}

\theoremstyle{definition}
\newtheorem{defin}[thm]{Definition}

\theoremstyle{remark}
\newtheorem{remark}[thm]{Remark}

\numberwithin{equation}{section}

\title[Potentially Good Reduction]
{A Criterion for Potentially Good Reduction
in Non-archimedean Dynamics}
\date{November 30, 2013}
\subjclass[2010]{Primary: 37P05 Secondary: 37P20, 11S82}
\keywords{arithmetic dynamics, good reduction, periodic points}

\author{Robert~L. Benedetto}
\address{Amherst College \\ Amherst, MA 01002}
\email{rlb@math.amherst.edu}

\begin{document}

\newcounter{bean}
\newcounter{sheep}

\begin{abstract}
Let $K$ be a non-archimedean field, and
let $\phi\in K(z)$ be a polynomial or rational function of degree
at least $2$.
We present a necessary and sufficient condition,
involving only the fixed points of $\phi$ and their preimages,
that determines whether or not
the dynamical system $\phi:\PP^1\to\PP^1$ has potentially
good reduction.
\end{abstract}

\maketitle

Fix the following notation throughout this paper.
\begin{tabbing}
\hspace{1cm} \= \hspace{2cm} \=  \kill
\> $K$: \> a field \\
\> $\overline{K}$: \> an algebraic closure of $K$ \\
\> $|\cdot|_v$: \> a non-archimedean absolute value on $\overline{K}$ \\
\> $\ints_K$: \> the ring of integers $\{x\in K : |x|_v \leq 1\}$ of $K$ \\
\> $\maxid_K$: \> the maximal idea $\{x\in K : |x|_v < 1\}$ of $\ints_K$ \\
\> $k$: \> the residue field $\ints_K/ \maxid_K$ of $K$ \\
\> $\ints_{\overline{K}}$: \> the ring of integers of $\overline{K}$ \\
\> $\maxid_{\overline{K}}$: \> the maximal idea of $\ints_{\overline{K}}$ \\
\> $\overline{k}$: \> the residue field of $\overline{K}$
\end{tabbing}
For example, $K$ could be the field $\Qp$ of $p$-adic rationals,
with ring of integers $\Zp$, maximal ideal $p\Zp$, and algebraic
closure $\Qpbar$.  Since $\Qp$ is complete, the absolute value
$|\cdot|_p$ on $\Qp$ extends uniquely to $\Qpbar$.

Let $\phi(z)\in K(z)$ be a rational function.
We define the degree of $\phi=f/g$ to be
$\deg\phi:=\max\{\deg f, \deg g\}$,
where $f,g\in K[z]$ have no common factors.
We will view $\phi$ as a a dynamical system acting on
$\PKbar=\overline{K}\cup\{\infty\}$.
For a thorough treatment of dynamics over such non-archimedean fields,
see \cite{BR}.

If $\deg\phi\geq 2$,
then $\phi$ fixes exactly $1+\deg\phi$ points
in $\PKbar$, counted with appropriate multiplicity.
The \emph{multiplier} $\lambda\in\overline{K}$ of
such a fixed point $x\in\PKbar$ is
\begin{equation}
\label{eq:multdef}
\lambda:=\begin{cases}
\phi'(x) & \text{ if } x\in\overline{K},
\\
(\psi)'(0) & \text{ if } x=\infty, \text{ where }\psi(z)=1/\phi(1/z).
\end{cases}
\end{equation}
If $|\lambda|_v>1$, we say the fixed point $x$ is \emph{repelling};
if $|\lambda|_v=1$, we say $x$ is \emph{indifferent};
and if $|\lambda|_v<1$, we say $x$ is \emph{attracting}.
The multiplier of $x$ is exactly $1$ (as opposed to merely
having absolute value $1$) if and only if $x$ has multiplicity
at least two as a fixed point.
It is $0$ if and only if $x$ is a critical point of $\phi$.

If we change coordinates on $\PKbar$ by applying
a linear fractional transformation $h\in\PGL(2,\overline{K})$,
the effect on $\phi$ is to conjugate by $h$.
In particular,
if $x$ is a fixed point of $\phi$, then $h(x)$ is a fixed point of
$h\circ\phi\circ h^{-1}\in\Kbar(z)$,
and with the same multiplier.
This invariance justifies the definition of the multiplier of
a fixed point at $\infty$ in equation~\eqref{eq:multdef}.

The notion of good reduction of
$\phi\in K(z)$ was first stated in \cite{MS1};
see Definition~\ref{def:goodred}.
We will say that $\phi$ has \emph{potentially good reduction}
if $\phi$ is conjugate over $\Kbar$ to a map of good reduction.
It is easy to see whether or not a given map $\phi\in K(z)$
has good reduction as written, but in general,
it is not so obvious how to tell
whether or not $\phi$ is conjugate to a map of
good reduction.

In \cite{Rum}, Rumely presents two somewhat involved algorithms
for deciding whether or not $\phi$ has potentially good reduction.
A much simpler necessary and sufficient condition for potential
good reduction when $\phi\in K[z]$ is a \emph{polynomial} appeared
in Corollary~4.6 of \cite{Ben4}.
In this paper, we generalize that simpler criterion to apply
to arbitrary rational functions $\phi\in K(x)$, as follows.

\begin{thmblank}
Let $K$ be a field with algebraic closure $\overline{K}$
and with non-archimedean absolute value $|\cdot|_v$ on $\overline{K}$.
Let $\phi\in K(z)$ be a rational function of degree $d\geq 2$.
Let $x_1,\ldots, x_{d+1}\in\PKbar$ be the fixed points of $\phi$,
repeated according to multiplicity.
If any $x_i$ is repelling,
then $\phi$ does not have potentially good reduction.
Otherwise, define $h\in\PGL(2,\Kbar)$ as follows.
\begin{enumerate}
\item
If $x_i$ is indifferent for some $i$, then
it is possible to choose $y_1\in \phi^{-1}(x_i)\smallsetminus\{x_i\}$,
and $y_2\in\phi^{-1}(y_1)$.
The three points $x_i, y_1, y_2\in\PKbar$ are necessarily distinct, and so
there is a unique $h\in\PGL(2,\overline{K})$ such that
$h(x_i)=0$, $h(y_1)=1$, and $h(y_2)=\infty$.
\item
Otherwise, all the fixed points are attracting,
and $x_1$, $x_2$, and $x_3$ are necessarily distinct.
Thus, there is a unique $h\in\PGL(2,\overline{K})$ such that
$h(x_1)=0$, $h(x_2)=1$, and $h(x_3)=\infty$.
\end{enumerate}
Let $\psi = h\circ\phi\circ h^{-1}$.
The original map $\phi$ has potentially good reduction
if and only if $\psi$ has good reduction; in that case,
$h$ provides the desired coordinate change.
\end{thmblank}

If $\phi$ does have potentially good reduction, the map
$h\in\PGL(2,\Kbar)$ of the above Theorem is defined over
an extension of degree at most $d^3-d$ over $K$, \emph{a priori}.
After all, the fixed points of $\phi$ are all roots of a polynomial
of degree $d+1$, and the preimages of any given point of
$\PKbar$ are roots of a polynomial of degree $d$.  In
case (a) above, then, $x_i$ contributes degree up to $d+1$,
$y_1$ contributes up to $d-1$, and $y_2$ contributes up to $d$.
Similarly, in case (b), $x_0$, $x_1$, and $x_2$
contribute up to $d+1$, $d$, and $d-1$.
However, in a separate paper \cite{BenPG2},
we will show that this bound
may be reduced from $d^3-d$ to $\max\{d+1,d^2-d\}$.

We will recall some general facts about good reduction in
Section~\ref{sec:reduc}.  Then, in Section~\ref{sec:proof},
we will state some simple lemmas and prove the Theorem.

\section{Reduction of rational maps}
\label{sec:reduc}

Given a polynomial $f(z)\in \ints_{\overline{K}}[z]$, denote by
$\overline{f}(z)\in \overline{k}[z]$ the polynomial formed by
reducing all coefficients of $f$ modulo $\maxid_{\overline{K}}$.

\begin{defin}
\label{def:goodred}
Let $\phi(z)\in K(z)$ be a rational function
of degree $d\geq 1$.  Write $\phi=f/g$ with
$f,g\in\ints_K[z]$ and with at least one coefficient
of $f$ or $g$ having absolute value $1$.
Let $\overline{\phi}:=\overline{f}/\overline{g}$.
We say that $\phi$ has {\em good reduction}
if $\deg \overline{\phi} = \deg \phi$.
Otherwise,
we say that $\phi$ has \emph{bad reduction}.

We say that $\phi$ has \emph{potentially good reduction}
if there is some $h\in\PGL(2,\overline{K})$ such that
$h\circ\phi\circ h^{-1}\in\overline{K}(z)$ has good reduction.
\end{defin}

There is some inconsistency in the literature as to whether
``bad reduction'' of a dynamical system $\phi$ should mean
the lack of good reduction, or, more restrictively, the lack
of good reduction even after a $K$-rational change of coordinates.
For example, if $a\in K$ is a uniformizer, then the map
$\phi(z) = az^2$ has bad reduction as written, but its
$K$-rational conjugate $a\phi(a^{-1}z)=z^2$ has good reduction.
Meanwhile, the map $\psi(z)=az^3$ also has bad reduction,
and in fact it can be shown to have bad reduction even after
any $K$-rational coordinate change.  However, $\psi$ has
potentially good reduction, because
$a^{1/2}\psi(a^{-1/2}z) = z^3$ has good reduction.

Given $\phi\in K(z)$, it is easy to check that polynomials
$f,g\in\ints_K[z]$ exist as specified in Definition~\ref{def:goodred},
and that the reduction type of $\phi$ does not change if
two other such polynomials
$\tilde{f},\tilde{g}\in\ints_K[z]$ are chosen in their stead.
Note that $\deg\overline{\phi}\leq \deg\phi$, but 
the degree could drop if the formerly relatively
prime $f,g\in \ints_K[z]$ gain a common factor when
they are reduced, or if
$\max\{\deg\overline{f},\deg\overline{g}\}$
is strictly smaller than $\max\{\deg f,\deg g\}$.

The reduction map $\red:\ints_{\overline{K}}\to \overline{k}$
induces a map $\red:\PKbar\to\Pkbar$, which coincides with the
original reduction map on $\ints_{\overline{K}}$ and maps
$\PKbar\smallsetminus\ints_{\overline{K}}$ to the point
$\overline{\infty}\in\Pkbar$.  It is easy to check that
a rational function $\phi\in\Kbar(z)$ has good reduction
if and only if it respects the reduction map, i.e., if
$$\overline{\phi}\big(\bar{x}\big) = \overline{\phi(x)}
\quad\text{for all}\quad x\in\PKbar.$$
As a result, the composition of two maps of good reduction
again has good reduction.
In addition, among maps of degree~1, i.e., in $\PGL(2,\Kbar)$,
the good reduction maps are precisely the elements of
$\PGL(2,\ints_{\Kbar})$.


\begin{lemma}
\label{lem:multabs}
Let $\psi\in\Kbar(z)$ have good reduction, and write $\psi=f/g$ 
with $f,g\in \ints_{\Kbar}[z]$ as in Definition~\ref{def:goodred}.
Suppose that $0$ is a fixed point of $\psi$,
and let $\lambda$ be the associated multiplier.
Then $f(0)=0$, $|g(0)|_v = 1$, and $|\lambda|_v = |f'(0)|_v$.
\end{lemma}

\begin{proof}
Since $f$ and $g$ are relatively prime, we have
$f(0)=0$ and $g(0)\neq 0$.  However, because
$\bar{f},\bar{g}\in\ints_{\Kbar}[z]$ are also relatively prime,
we must in fact have $|g(0)|_v=1$.  Thus,
\[
|\lambda|_v = \Big|\frac{f'(0)g(0) - f(0)g'(0)}{g(0)^2}\Big|_v
= \frac{|f'(0) g(0)|_v}{|g(0)|_v^2} = |f'(0)|_v .
\qedhere
\]
\end{proof}

\section{Proof of the Theorem}
\label{sec:proof}

The following lemmas are well known in the more general setting
of dynamics on a non-archimedean disk, but their proofs in our
context are very short and simple, and so we include
them for the convenience of the reader.

\begin{lemma}
\label{lem:norepgood}
Let $\psi\in \Kbar(z)$ be a rational function
of good reduction, and let $x\in\PKbar$ be a fixed point.
Then $x$ is indifferent or attracting.
\end{lemma}

\begin{proof}
There is some degree one map $h\in\PGL(2,\ints_{\Kbar})$
of good reduction such that $h(x)=0$.  Since conjugating by $h$
preserves good reduction and multipliers, we may assume without
loss that $x=0$.
Let $\lambda:=\psi'(0)$ be the associated multiplier.

Write $\psi=f/g$ as in Definition~\ref{def:goodred}.
Then $f'\in\ints_{\Kbar}[z]$, since $f\in\ints_{\Kbar}[z]$.
Thus, by Lemma~\ref{lem:multabs}, we have
$|\lambda|_v = |f'(0)|_v\leq 1$.
\end{proof}

\begin{lemma}
\label{lem:attrgood}
Let $\psi\in\Kbar(z)$ be a rational function
of good reduction, and let $x\in\PKbar$ be an attracting fixed point
of $\psi$.
Then for any other fixed point $y\in\PKbar$ of $\psi$,
the reductions $\red(x),\red(y)\in\Pkbar$ are distinct.
\end{lemma}

\begin{proof}
By a $\PGL(2,\ints_{\Kbar})$ change of coordinates, we may assume that
$x=0$.  Write $\psi=f/g$ as in Definition~\ref{def:goodred},
and let $\lambda:=\psi'(0)$ be the associated multiplier.
By Lemma~\ref{lem:multabs}, we have $f(0)=0$,
$|g(0)|_v=1$, and $|f'(0)|_v=|\lambda|_v<1$.  That is,
$f(z) = a_1 z + \cdots + a_d z^d$ with $|a_i|_v\leq 1$
and $|a_1|_v<1$, whereas
$g(z) = b_0 + b_1 z + \cdots + b_d z^d$ with $|b_i|_v\leq 1$
and $|b_0|_v=1$.

For any $y\in\PKbar$ with $\red(y)=\red(0)$, we have $|y|_v<1$.  Thus,
if $y\neq 0$, we have
$$|\psi(y)|_v
= |y|_v\cdot \frac{|a_1 + \cdots + a_d y^{d-1}|_v}{|b_0 + \cdots + b_d y^d|_v}
= |y|_v\cdot \frac{|a_1 + \cdots + a_d y^{d-1}|_v}{|b_0|_v}
< |y|_v,$$
and hence $y$ cannot be fixed.
\end{proof}

\begin{lemma}
\label{lem:indiffgood}
Let $\psi\in\Kbar(z)$ be a rational function
of good reduction and degree $d\geq 2$,
and let $x\in\PKbar$ be an indifferent fixed point
of $\psi$.  Then $\psi^{-1}(x)\smallsetminus\{x\}$
is nonempty, and $\red(x)\neq \red(y)$
for any $y\in\psi^{-1}(x)\smallsetminus\{x\}$.
\end{lemma}

\begin{proof}
By a $\PGL(2,\ints_{\Kbar})$ change of coordinates, we may assume that
$x=0$.  Write $\psi=f/g$ as in Definition~\ref{def:goodred},
and let $\lambda:=\psi'(0)$ be the associated multiplier.
By Lemma~\ref{lem:multabs}, we have $f(0)=0$,
$|g(0)|_v=1$, and $|f'(0)|_v=|\lambda|_v=1$.  Thus,
$f(z) = a_1 z + \cdots + a_d z^d$ with $|a_i|_v\leq 1$
and $|a_1|_v=1$.
It is immediate that $f$, and hence $\psi$, has no roots in
$\calM_{\Kbar}\smallsetminus\{0\}$.

Meanwhile, the hypothesis on $\psi'(0)$ also implies that
$0$ is not a critical point.
Since $d\geq 2$, $\psi^{-1}(0)\smallsetminus\{0\}$ is nonempty.
Given $y\in\psi^{-1}(0)\smallsetminus\{0\}$,
the previous paragraph shows that $y\not\in\calM_{\Kbar}$,
and hence $\red(y)\neq\red(0)$.
\end{proof}

\begin{remark}
Lemmas~\ref{lem:attrgood} and~\ref{lem:indiffgood} are special
cases more general statements about fixed points of a map
$\phi\in K(z)$ of good reduction.  Specifically, if $x$ is
an attracting fixed point (or more generally, periodic point)
of such a map, then the residue class of $x$ contains no other
periodic points, although it almost certainly contains many
other \emph{preperiodic} points that ultimately map to $x$
itself.  Similarly, if $x$ is an indifferent periodic point
of a map of good reduction, then the residue class of
$x$ maps bijectively onto itself, and hence the only preperiodic
points in the residue class are in fact periodic.
(In fact, there are usually infinitely many periodic points in such a
residue class.)

Furthermore, these facts about preperiodic points near nonrepelling periodic
points were generalized by Rivera-Letelier in his thesis \cite{Riv1}
to maps of bad reduction, by replacing residue 
classes by an appropriate notion of Fatou components.
However, we will only need the far simpler statements of
Lemmas~\ref{lem:attrgood} and~\ref{lem:indiffgood} here.
\end{remark}

\begin{proof}[Proof of the Theorem]
For each $i=1,\ldots,d+1$, let $\lambda_i\in\Kbar$ be the
multiplier of the fixed point $x_i$.
If $\phi$ has potentially good reduction, then there is some
$h\in\PGL(2,\overline{K})$ such that $h\circ\phi\circ h^{-1}$
has good reduction.  The fixed points of
$h\circ\phi\circ h^{-1}$ have the same multipliers as those
of $\phi$, and therefore, by Lemma~\ref{lem:norepgood},
we have $|\lambda_i|_v \leq 1$ for all $i$, proving the
first statement of the Theorem.
We now consider two cases.

\textbf{Case 1}.
If $|\lambda_i|_v=1$ for some $i$, then the associated fixed point
$x_i$ is not a critical point (since $\lambda_i\neq 0$), and hence
$\phi^{-1}(x_i)\smallsetminus\{x_i\}$ is nonempty.  Pick 
$y_1\in \phi^{-1}(x_i)\smallsetminus\{x_i\}$, and pick
$y_2\in \phi^{-1}(y_1)$.  Then $y_1\neq x_i$ by definition,
and hence $y_2\neq y_1$ also.  Moreover, $y_2\neq x_i$, because
if $y_2=x_i$, then taking $\phi$ of both gives $y_1=x_i$.
Let $h$ be the unique map in $\PGL(2,\Kbar)$ with
$h(x_i)=0$, $h(y_1)=1$, and $h(y_2)=\infty$,
and let $\psi:=h\circ\phi\circ h^{-1}$
If $\psi$ has good reduction, then $\phi$ has
potentially good reduction, and we are done.

Conversely, if $\phi$ has potentially good reduction,
there is some $\tilde{h}\in\PGL(2,\Kbar)$ such that
$\tilde{\psi}:=\tilde{h} \circ\phi\circ\tilde{h}^{-1}$
has good reduction.  
By Lemma~\ref{lem:indiffgood}, the indifferent fixed point
$\tilde{h}(x_i)$ of $\tilde{\psi}$ and its preimage
$\tilde{h}(y_1)$ (under $\tilde{\psi}$) have different reductions
in $\Pkbar$.  Since $\tilde{\psi}$ respects reduction,
we must have $\bar{\tilde{\psi}}(\overline{\tilde{h}(y_2)})
= \overline{\tilde{h}(y_1)}$, and for the same reasons as in the previous
paragraph, the reductions
$\tilde{h}(x_i)$, $\tilde{h}(y_1)$, and $\tilde{h}(y_2)$
must all be distinct points in $\Pkbar$.  Thus, there is
some map $r\in\PGL(2,\ints_{\Kbar})$ such that
$$r\big(\tilde{h}(x_i)\big) = 0,
\quad
r\big(\tilde{h}(y_1)\big) = 1,
\quad\text{and}\quad
r\big(\tilde{h}(y_2)\big)=\infty.$$
Then we must have $r\circ \tilde{h}=h$,
and $\psi=r\circ\tilde{\psi}\circ r^{-1}$.
Since $\tilde{\psi}$, $r$, and $r^{-1}$ all have good
reduction, it follows that $\psi$ also has good reduction.

\textbf{Case 2}.
The remaining case is that no $x_i$ is repelling or indifferent,
and hence they are all attracting.  That is,
$|\lambda_i|_v<1$ for all $i$.  Then no $\lambda_i$
is $1$, and hence the $d+1$ fixed points all
have multiplicity $1$ and therefore are distinct.
In particular, $x_1$, $x_2$, and $x_3$ are all distinct.
Let $h$ be the unique map in $\PGL(2,\Kbar)$ with
$h(x_1)=0$, $h(x_2)=1$, and $h(x_3)=\infty$,
and let $\psi:=h\circ\phi\circ h^{-1}$.
If $\psi$ has good reduction, then $\phi$ has
potentially good reduction, and we are done.

Conversely, if $\phi$ has potentially good reduction,
there is some $\tilde{h}\in\PGL(2,\Kbar)$ such that
$\tilde{\psi}:=\tilde{h} \circ\phi\circ\tilde{h}^{-1}$
has good reduction.  
By Lemma~\ref{lem:attrgood}, the attracting fixed points
$\tilde{h}(x_1)$, $\tilde{h}(x_2)$, and $\tilde{h}(x_3)$
of $\tilde{\psi}$ all have different reductions
in $\Pkbar$.  Thus, there is
some map $r\in\PGL(2,\ints_{\Kbar})$ such that
$$r\big(\tilde{h}(x_1)\big) = 0,
\quad
r\big(\tilde{h}(x_2)\big) = 1,
\quad\text{and}\quad
r\big(\tilde{h}(x_3)\big)=\infty.$$
Then we must have $r\circ \tilde{h}=h$,
and $\psi=r\circ\tilde{\psi}\circ r^{-1}$.
Since $\tilde{\psi}$, $r$, and $r^{-1}$ all have good
reduction, it follows that $\psi$ also has good reduction.
\end{proof}

{\bf Acknowledgements.}
The author thanks Xander Faber and Joseph Silverman for their
helpful suggestions.
The author also
gratefully acknowledges the support of NSF grant DMS-1201341.


\begin{thebibliography}{99}


\bibitem{BR}
Matthew H.~Baker and Robert S.~Rumely,
{\em Potential Theory and Dynamics on the Berkovich Projective Line,}
Amer.\ Math.\ Soc., Providence, 2010.


\bibitem{Ben4}
R.~Benedetto,
\newblock{Reduction, dynamics, and Julia sets of rational functions,}
\newblock{{\em J. Number Theory} {\bf 86} (2001), 175--195.}

\bibitem{BenPG2}
R.~Benedetto,
\newblock{Attaining potentially good reduction in arithmetic dynamics,}
\newblock{in preparation.}



	
















\bibitem{MS1}
P.~Morton and J.~Silverman,
\newblock{Rational periodic points of rational functions,}
\newblock{{\em Inter. Math. Res. Notices\/} {\bf 2} (1994), 97--110.}





\bibitem{Riv1}
J.~Rivera-Letelier,
\newblock{Dynamique des fonctions rationnelles sur des corps locaux,}
\newblock{{\em Ast\'erisque} {\bf 287} (2003), 147--230.}


\bibitem{Rum}
R.~Rumely,
\newblock{The minimal resultant locus,}
\newblock{preprint.  Available at \texttt{arXiv:1304.1201}.}







\end{thebibliography}
\end{document}